\documentclass[12pt,letterpaper,oneside,american,mathcal,twoside]{amsart}
\usepackage{bookman}
\usepackage{helvet}
\usepackage{courier}
\usepackage[T1]{fontenc}
\usepackage[latin9]{inputenc}
\usepackage{amsthm}
\usepackage{amstext}
\usepackage{amssymb}
\usepackage[all]{xy}

\makeatletter

\pdfpageheight\paperheight
\pdfpagewidth\paperwidth

\numberwithin{equation}{section}
\numberwithin{figure}{section}
\theoremstyle{plain}
\newtheorem{thm}{\protect\theoremname}[section]
  \theoremstyle{definition}
  \newtheorem{defn}[thm]{\protect\definitionname}
  \theoremstyle{remark}
  \newtheorem{rem}[thm]{\protect\remarkname}
  \theoremstyle{remark}
  \newtheorem*{rem*}{\protect\remarkname}
  \theoremstyle{definition}
  \newtheorem{example}[thm]{\protect\examplename}
  \theoremstyle{plain}
  \newtheorem{prop}[thm]{\protect\propositionname}
  \theoremstyle{plain}
  \newtheorem{lem}[thm]{\protect\lemmaname}
  \theoremstyle{plain}
  \newtheorem{cor}[thm]{\protect\corollaryname}

\usepackage{verbatim} 
\usepackage{varioref}
\usepackage{eucal}
\usepackage{amsthm}
\usepackage[all]{xy}
\usepackage{rotating}
\usepackage{psfrag}
\usepackage{float}
\usepackage{xr}
\usepackage{eucal}
\usepackage{dsfont}

\SelectTips{cm}{}
\xyoption{line}

\SelectTips{cm}{}
\xyoption{line}

\makeatletter \newcommand{\xyR}[1]{%
\makeatletter \xydef@\xymatrixrowsep@{#1} \makeatother }
\makeatletter \newcommand{\xyC}[1]{%
\makeatletter \xydef@\xymatrixcolsep@{#1} \makeatother }
\mathcode`\:="603A % Redefine ':' to be math punctuation since I never use
\DeclareSymbolFont{rsfs}{U}{rsfs}{m}{n}
\DeclareSymbolFontAlphabet{\mathrf}{rsfs}

\makeatother

\usepackage{babel}
  \providecommand{\corollaryname}{Corollary}
  \providecommand{\definitionname}{Definition}
  \providecommand{\examplename}{Example}
  \providecommand{\lemmaname}{Lemma}
  \providecommand{\propositionname}{Proposition}
  \providecommand{\remarkname}{Remark}
\providecommand{\theoremname}{Theorem}

\begin{document}

\title{Cellular coalgebras over the Barratt-Eccles operad I.}

\author{Justin R. Smith}

\date{\today}



\email{jsmith@drexel.edu}

\urladdr{http://vorpal.math.drexel.edu }

\address{Department of Mathematics\\
Drexel University\\
Philadelphia, PA 19104}

\maketitle
\global\long\def\ring{\mathbb{Z}}
\global\long\def\integers{\mathbb{Z}}
\global\long\def\coend{\mathrm{CoEnd}}
 \global\long\def\coassoc{\mathrm{Coassoc}}
\global\long\def\homz{\mathrm{Hom}_{\ring}}
\global\long\def\homa{\mathrm{Hom}}
\global\long\def\zend{\mathrm{End}}
\global\long\def\rs#1{\mathrm{R}S_{#1 }}
\global\long\def\forgetful#1{\lceil#1\rceil}
\global\long\def\zs#1{\mathbb{Z}S_{#1 }}
\global\long\def\homzs#1{\mathrm{Hom}_{\ring S_{#1 }}}
\global\long\def\rings#1{\ring S_{#1}}
\global\long\def\zpi{\mathbb{Z}\pi}
\global\long\def\cf#1{\mathcal{C}(#1 )}
\global\long\def\ddelta{\dot{\Delta}}
\global\long\def\dimlimiter{\triangleright}
\global\long\def\coalgcat{\mathrf S_{0}}
\global\long\def\ccoalgcat{\mathrf S_{\mathrm{cell}}}
\global\long\def\hcoalgcat{\mathrf{S}}
\global\long\def\ircoalgcat{\mathrf I_{0}}
\global\long\def\bircoalgcat{\mathrf{I}_{0}^{+}}
\global\long\def\hircoalgcat{\mathrf I}
\global\long\def\chaincat{\mathbf{Ch}_{0}}
\global\long\def\coll{\mathrm{Coll}}
\global\long\def\ilimit{\varprojlim\,}
\global\long\def\dlimit{\varinjlim\,}
\global\long\def\dilimit{{\varprojlim}^{1}\,}
\global\long\def\spaces{\mathbf{S}_{0}}

\global\long\def\coker{\mathrm{{coker}}}
\global\long\def\lcell{L_{\mathrm{cell}}}
\global\long\def\ccoalgcat{\mathrf S_{\mathrm{cell}}}
\global\long\def\fc#1{\mathrm{hom}(\bigstar,#1)}
\global\long\def\coS{\mathbf{coS}}
\global\long\def\cocell{\mathbf{co}\ccoalgcat}
\global\long\def\realization{\mathcal{H}}
\global\long\def\totalspace#1{\mathrm{Tot}(#1)}
\global\long\def\totalcoalg#1{\mathrm{Tot}_{\ccoalgcat}(#1)}
\global\long\def\cofc#1{\mathbf{co}\fc{#1}}
\global\long\def\coforgetful#1{\mathbf{co}-\forgetful{#1}}
\global\long\def\cochaincat{\mathbf{co}\chaincat}
\global\long\def\zinfty#1{\integers_{\infty}#1}
\global\long\def\pgam{\tilde{\Gamma}}
\global\long\def\pz{\tilde{\integers}}
\global\long\def\sab{\mathbf{sAB}_{0}}
\global\long\def\homzp{\mathrm{Hom}_{\integers[C_{p}]}}
\global\long\def\S{\mathfrak{S}}
\global\long\def\finitecofree{\mathcal{F}_{\S}^{c}}
\global\long\def\rzp{\mathrm{R}(\integers/p\cdot\integers)}
\global\long\def\vp{\mathcal{V}_{p}}
\global\long\def\trivialcofree{\bar{\mathcal{F}_{\S}}}
\global\long\def\slength#1{|#1|}
\global\long\def\barcs{\bar{\mathcal{B}}}

\global\long\def\exp{\operatorname{exp}}

\global\long\def\ints{\mathbb{Z}}

\global\long\def\rats{\mathbb{Q}}

\global\long\def\qhat{\hat{Q}}

\global\long\def\moore#1{\{#1\}}
 
\begin{abstract}
This paper considers a class of coalgebras over the Barratt-Eccles
operad and shows that they classify $\ints$-completions of pointed,
reduced simplicial sets. As a consequence, they encapsulate the homotopy
types of nilpotent simplicial sets. This result is a direct generalization
of Quillen's result characterizing rational homotopy types via cocommutative
coalgebras.
\end{abstract}

\section{Introduction}

In \cite{Smith:1994} the author constructed coalgebra-structures
over the Barratt-Eccles operad on the integral chain-complex of a
pointed, reduced simplicial set. These coalgebra structures suffice
to compute all Steenrod operations (among other things). The present
paper shows that such coalgebras have an algebraic property called
cellularity (see definition~\ref{def:cellularcoalgebra}). Cellular
coalgebras are shown to have an analogue of the Hurewicz map (corollary~\ref{cor:classifying-map-space-hurewicz})
that allows us to construct their ``$\ints$-completion''. In the
case that a cellular coalgebra is topologically realizable, properties
of the space's $\ints$-completion can be derived from that of the
coalgebra.

The main technical results, theorem~\ref{thm:injectivity-theorem}
and corollary~\ref{cor:classifying-map-space-hurewicz}, imply that
cellular coalgebras have a ``Hurewicz map'' that precisely corresponds
to the topological Hurewicz map when a cellular coalgebra is topologically
realizable.

We use this to derive a cosimplicial resolution of a cellular coalgebra
and show that it is weakly equivalent to the $\ints^{\bullet}$-resolution
when the coalgebra is derived from a topological space.

Our final result is a kind of generalization of Quillen's main result
(in \cite{Quillen:1969}) characterizing rational homotopy types via
commutative coalgebras:

\medskip 

\emph{Corollary~\ref{cor:cf-classifies-nilpotent}: If $X$ and $Y$
are pointed nilpotent reduced simplicial sets, then $X$ is weakly
equivalent to $Y$ if and only if there exists a morphism of cellular
coalgebras
\[
f:\cf X\to\cf Y
\]
that is an integral homology equivalence.}

\medskip 

An earlier somewhat similar statement was published by Smirnov in
\cite{Smirnov:1985}, but his proof was unclear and he used an operad
that was uncountably generated in all dimensions. Smirnov's proof
was so unclear that several people known to the author believed the
result to be untrue. 

The the present paper's proof is a straightforward application of
simplicial resolutions --- involving the operad used to compute Steenrod
operations.

The reader might wonder why our coalgebras seem to encode more information
than structures nominally dual to them, like algebras. The answer
is that \emph{nilpotent} coalgebras%
\footnote{Roughly speaking, nilpotent coalgebras are ones in which iterated
coproduct ``peter out'' after a finite number of steps --- see \cite[chapter~3]{operad-book}
for the precise definition.%
} are dual to algebras --- and the coalgebras we consider are \emph{not}
nilpotent (see proposition~\ref{pro:simplicespropertyS}). The paper
\cite{Smith:cofree} showed that cofree coalgebras (see definition~\ref{def:cofreecoalgebra})
are \emph{not} duals to free \emph{algebras} --- they are somewhat
like ``profinite completions'' of them.

The duals of the coalgebras considered here are algebraic structures
that ``look like'' algebras but have the property that certain ``infinite
products'' are well-defined. If one avoids taking such infinite products
one gets the usual cohomology algebra that can be used to define Steenrod
operations, etc. This involves throwing out significant ``transcendental''
data.

\section{Definitions and assumptions}

We will denote the closed symmetric monoidal category of $\ring$-free
$\ring$-chain-complexes concentrated in positive dimensions by $\chaincat$. 

We make extensive use of the Koszul Convention (see~\cite{Gugenheim:1960})
regarding signs in homological calculations:
\begin{defn}
\label{def:koszul-1} If $f:C_{1}\to D_{1}$, $g:C_{2}\to D_{2}$
are maps, and $a\otimes b\in C_{1}\otimes C_{2}$ (where $a$ is a
homogeneous element), then $(f\otimes g)(a\otimes b)$ is defined
to be $(-1)^{\deg(g)\cdot\deg(a)}f(a)\otimes g(b)$. \end{defn}
\begin{rem}
If $f_{i}$, $g_{i}$ are maps, it isn't hard to verify that the Koszul
convention implies that $(f_{1}\otimes g_{1})\circ(f_{2}\otimes g_{2})=(-1)^{\deg(f_{2})\cdot\deg(g_{1})}(f_{1}\circ f_{2}\otimes g_{1}\circ g_{2})$.\end{rem}
\begin{defn}
\label{def:homcomplex-1}Given chain-complexes $A,B\in\chaincat$
define
\[
\homz(A,B)
\]
to be the chain-complex of graded $\ring$-morphisms where the degree
of an element $x\in\homz(A,B)$ is its degree as a map and with differential
\[
\partial f=f\circ\partial_{A}-(-1)^{\deg f}\partial_{B}\circ f
\]
As a $\ring$-module $\homz(A,B)_{k}=\prod_{j}\homz(A_{j},B_{j+k})$.\end{defn}
\begin{rem*}
Given $A,B\in\mathbf{Ch}^{S_{n}}$, we can define $\homzs n(A,B)$
in a corresponding way.\end{rem*}
\begin{defn}
\label{def:tmap}Let $\sigma\in S_{n}$ be an element of the symmetric
group and let $\{k_{1},\dots,k_{n}\}$ be $n$ nonnegative integers
with $K=\sum_{i=1}^{n}k_{i}$. Then $T_{k_{1},\dots,k_{n}}(\sigma)$
is defined to be the element $\tau\in S_{K}$ that permutes the $n$
blocks 
\[
(1,\dots,k_{1}),(k_{1}+1,\dots,k_{1}+k_{2})\dots(k-k_{n-1},\dots,k)
\]
as $\sigma$ permutes the set $\{1,\dots,n\}$.\end{defn}
\begin{rem}
Note that it is possible for one of the $k$'s to be $0$, in which
case the corresponding block is empty.\end{rem}
\begin{defn}
\label{def:operad}If $G$ is a discrete group, let $\chaincat^{G}$
denote the category of chain-complexes equipped with a right $G$-action.
This is again a closed symmetric monoidal category and the forgetful
functor $\chaincat^{G}\to\chaincat$ has a left adjoint, $(-)[G]$.
This applies to the symmetric groups, $S_{n}$, where we regard $S_{1}$
and $S_{0}$ as the trivial group. The \emph{category of collections}
is defined to be the product
\[
\mathrm{Coll}(\chaincat)=\prod_{n\ge0}\chaincat^{S_{n}}
\]
Its objects are written $\mathcal{V}=\{\mathcal{V}(n)\}_{n\ge0}$.
Each collection induces an endofunctor (also denoted $\mathcal{V}$)
$\mathcal{V}:\chaincat\to\chaincat$
\[
\mathcal{V}(X)=\bigoplus_{n\ge0}\mathcal{V}(n)\otimes_{\zs n}X^{\otimes n}
\]
where $X^{\otimes n}=X\otimes\cdots\otimes X$ and $S_{n}$ acts on
$X^{\otimes n}$ by permuting factors. This endofunctor is a \emph{monad}
if the defining collection has the structure of an \emph{operad},
which means that $\mathcal{V}$ has a unit $\eta:\ring\to\mathcal{V}(1)$
and structure maps
\[
\gamma_{k_{1},\dots,k_{n}}:\mathcal{V}(n)\otimes\mathcal{V}(k_{1})\otimes\cdots\otimes\mathcal{V}(k_{n})\to\mathcal{V}(k_{1}+\cdots+k_{n})
\]
satisfying well-known equivariance, associativity, and unit conditions
--- see \cite{Smith:cofree}, \cite{Kriz-May}.

We will call the operad $\mathcal{V}=\{\mathcal{V}(n)\}$ $\Sigma$-\emph{cofibrant}
if $\mathcal{V}(n)$ is $\zs n$-projective for all $n\ge0$.\end{defn}
\begin{rem*}
The operads we consider here correspond to \emph{symmetric} operads
in \cite{Smith:cofree}.

The term ``unital operad'' is used in different ways by different
authors. We use it in the sense of Kriz and May in \cite{Kriz-May},
meaning the operad has a $0$-component that acts like an arity-lowering
augmentation under compositions. Here $\mathcal{V}(0)=\ring$.

The term $\Sigma$-\emph{cofibrant} first appeared in \cite{Berger-moerdijk-axiom-operad}.
\end{rem*}
A simple example of an operad is:
\begin{example}
\label{example:frakS0}For each $n\ge0$, $X$, $C(n)=\integers S_{n}$,
with structure-map induced by
\[
\gamma_{\alpha_{1},\dots,\alpha_{n}}:S_{n}\times S_{\alpha_{1}}\times\cdots\times S_{\alpha_{n}}\to S_{\alpha_{1}+\cdots+\alpha_{n}}
\]
defined by regarding each of the $S_{\alpha_{i}}$ as permuting elements
within the subsequence $\{\alpha_{1}+\cdots+\alpha_{i-1}+1,\dots,\alpha_{1}+\cdots+\alpha_{i}\}$
of the sequence $\{1,\dots,\alpha_{1}+\cdots+\alpha_{n}\}$ and making
$S_{n}$ permute these $n$-blocks. This operad is denoted $\mathfrak{S}_{0}$.
In other notation, its $n^{\text{th}}$ component is the \emph{symmetric
group-ring} $\integers S_{n}$. See \cite{Smith:1994} for explicit
formulas.
\end{example}
For the purposes of this paper, the main example of an operad is
\begin{defn}
\label{def:coend}Given any $C\in\chaincat$, the associated \emph{coendomorphism
operad}, $\coend(C)$ is defined by
\[
\coend(C)(n)=\homz(C,C^{\otimes n})
\]
 Its structure map
\begin{multline*}
\gamma_{\alpha_{1},\dots,\alpha_{n}}:\\
\homz(C,C^{\otimes n})\otimes\homz(C,C^{\otimes\alpha_{1}})\otimes\cdots\otimes\homz(C,C^{\otimes\alpha_{n}})\to\\
\homz(C,C^{\otimes\alpha_{1}+\cdots+\alpha_{n}})
\end{multline*}
simply composes a map in $\homz(C,C^{\otimes n})$ with maps of each
of the $n$ factors of $C$. 

This is a non-unital operad, but if $C\in\chaincat$ has an augmentation
map $\varepsilon:C\to\ring$ then we can  regard $\epsilon$ as the
only element of $\homz(C,C^{\otimes n})=\homz(C,C^{\otimes0})=\homz(C,\ring)$.
\end{defn}
Morphisms of operads are defined in the obvious way:
\begin{defn}
\label{def:operadmorphism} Given two operads $\mathcal{V}$ and $\mathcal{W}$,
a \emph{morphism} 
\[
f:\mathcal{V}\to\mathcal{W}
\]
 is a sequence of chain-maps 
\[
f_{i}:\mathcal{V}_{i}\to\mathcal{W}_{i}
\]
 commuting with all the diagrams in \ref{def:operad}.
\end{defn}
Verification that this satisfies the required identities is left to
the reader as an exercise.
\begin{defn}
\label{def:sfrakfirstmention}Let $\mathfrak{S}$ denote the Barratt-Eccles
operad with components $K(S_{n},1)$ --- the bar resolutions of $\integers$
over $\zs n$ for all $n>0$. See \cite{Smith:1994} for formulas
for the composition-operations. Coalgebras over $\mathfrak{S}$ are
chain-complexes equipped with a coassociative coproduct and Steenrod
operations for all primes (see \cite{Smith:1994}).\end{defn}
\begin{rem*}
The operad $\S$ was first described in \cite{Barratt-Eccles-operad}. \end{rem*}
\begin{defn}
\label{def:coalgebra-over-operad}A chain-complex $C$ is a \emph{coalgebra
over the operad} $\mathcal{V}$ if there exists a morphism of operads
\[
\mathcal{V}\to\coend(C)
\]

\end{defn}
The structure of a coalgebra over an operad can be described in several
equivalent ways:
\begin{enumerate}
\item $f_{n}:\mathcal{V}(n)\otimes C\to C^{\otimes n}$
\item $g:C\to\prod_{n=0}^{\infty}\homzs n(\mathcal{V}(n),C^{\otimes n})$
\end{enumerate}
where both satisfy identities that describe how composites of these
maps are compatible with the operad-structure.

\subsection{Types of coalgebras}
\begin{example}
Coassociative coalgebras are precisely the coalgebras over $\mathfrak{S}_{0}$
(see \ref{example:frakS0}). \end{example}
\begin{defn}
\label{def:coassoc}$\mathbf{Cocommut}$ is an operad defined to have
one basis element $\{b_{i}\}$ for each integer $i\ge0$. Here the
rank of $b_{i}$ is $i$ and the degree is 0 and the these elements
satisfy the composition-law: $\gamma(b_{n}\otimes b_{k_{1}}\otimes\cdots\otimes b_{k_{n}})=b_{K}$,
where $K=\sum_{i=1}^{n}k_{i}$. The differential of this operad is
identically zero. The symmetric-group actions are trivial.\end{defn}
\begin{example}
Coassociative, commutative coalgebras are the coalgebras over $\mathbf{Cocommut}$.
\end{example}
We will sometimes want to focus on a particular class of $\mathcal{V}$-coalgebras:
the \emph{pointed, irreducible coalgebras}. We define this concept
in a way that extends the conventional definition in \cite{Sweedler:1969}:
\begin{defn}
\label{def:pointedirreducible} Given a coalgebra over a unital operad
$\mathcal{V}$ with adjoint structure-maps
\[
a_{n}:C\to\homzs n(\mathcal{V}(n),C^{\otimes n})
\]
for $n\ge0$, an element $c\in C$ is called \emph{group-like} if
$a_{n}(c)=\homz(\epsilon_{n},1)(1\mapsto c^{\otimes n})$ for all
$n>0$. Here $c^{\otimes n}\in C^{\otimes n}$ is the $n$-fold $\ring$-tensor
product and $\epsilon_{n}:\mathcal{V}(n)\to\mathcal{V}(0)=\ring$
is the augmentation (which is $n$-fold composition with $\mathcal{V}(0)$). 

A coalgebra $C$ over an operad $\mathcal{V}$ is called \emph{pointed}
if it has a \emph{unique} group-like element (denoted $1$), and \emph{pointed
irreducible} if the intersection of any two sub-coalgebras contains
this unique group-like element.\end{defn}
\begin{rem*}
Note that a group-like element generates a sub $\mathcal{V}$-coalgebra
of $C$ and must lie in dimension $0$.

Although seemingly contrived, this condition arises in ``nature'':
The chain-complex of a pointed, reduced simplicial set is naturally
a pointed irreducible coalgebra over the Barratt-Eccles operad, $\mathfrak{S}=\{C(K(S_{n},1))\}$
(see \cite{Smith:1994}). \end{rem*}
\begin{prop}
Let $D$ be a pointed, irreducible coalgebra over an operad $\mathcal{V}$.
Then the augmentation map
\[
\varepsilon:D\to\ring
\]
is naturally split and any morphism of pointed, irreducible coalgebras
\[
f:D_{1}\to D_{2}
\]
 is of the form
\[
1\oplus\bar{f}:D_{1}=\ring\oplus\ker\varepsilon_{D_{1}}\to D_{2}=\ring\oplus\ker\varepsilon_{D_{2}}
\]
where $\varepsilon_{i}:D_{i}\to\ring$, $i=1,2$ are the augmentations.\end{prop}
\begin{proof}
The definition (\ref{def:pointedirreducible}) of the sub-coalgebra
$\ring\cdot1\subseteq D_{i}$ is stated in an invariant way, so that
any coalgebra morphism must preserve it. Any morphism must also preserve
augmentations because the augmentation is the $0^{\mathrm{th}}$-order
structure-map. Consequently, $f$ must map $\ker\varepsilon_{D_{1}}$to
$\ker\varepsilon_{D_{2}}$. The conclusion follows.\end{proof}
\begin{defn}
\label{def:pointedirredcat} We denote the \emph{category} of pointed
irreducible coalgebras over $\mathfrak{S}$ by $\coalgcat$. Every
such coalgebra, $C$, comes equipped with a canonical augmentation
\[
\varepsilon:C\to\ring
\]
so the \emph{terminal object} is $\ring$. If $\mathcal{V}$ is not
unital, the terminal object in this category is $0$, the null coalgebra.

The category of \emph{pointed irreducible coalgebras} over $\mathfrak{S}$
is denoted $\ircoalgcat$. Its terminal object is the coalgebra whose
underlying chain complex is $\ring$ concentrated in dimension $0$.
\end{defn}
We also need:
\begin{defn}
\label{def:forgetful}If $A\in\mathrf C=\ircoalgcat$ or $\coalgcat$,
then $\forgetful A$ denotes the underlying chain-complex in $\chaincat$
of
\[
\ker A\to t
\]
where $t$ denotes the terminal object in $\mathrf{C}$ --- see definition~
\ref{def:pointedirredcat}. We will call $\forgetful{\ast}$ the \emph{forgetful
functor} from $\mathrf{C}$ to $\chaincat$.
\end{defn}
We can define a concept dual to that of a free algebra generated by
a set: 
\begin{defn}
\label{def:cofreecoalgebra}Let $D$ be a coalgebra over an operad
$\mathfrak{S}$, equipped with a $\chaincat$-morphism $\varepsilon:\forgetful D\to E$,
where $E\in\chaincat$. Then $D$ is called \emph{the cofree coalgebra
over} $\mathfrak{S}$ \emph{cogenerated} \emph{by} $\varepsilon$
if any morphism in $\chaincat$
\[
f:\forgetful C\to E
\]
where $C\in\coalgcat$, induces a \emph{unique} morphism in $\coalgcat$
\[
\alpha_{f}:C\to D
\]
that makes the diagram 
\[
\xyC{40pt}\xymatrix{{\forgetful C}\ar[r]^{\forgetful{\alpha_{f}}}\ar[rd]_{f} & {\forgetful D}\ar[d]^{\varepsilon}\\
{} & {E}
}
\]
\emph{ }

Here $\alpha_{f}$ is called the \emph{classifying map} of $f$. If
$C\in\coalgcat$ then 
\[
\alpha_{f}:C\to L_{\mathfrak{S}}\forgetful C
\]
 will be called the \emph{classifying map of} $C$.
\end{defn}
This universal property of cofree coalgebras implies that they are
unique up to  isomorphism if they exist. 

The paper \cite{Smith:cofree} explicitly constructs cofree coalgebras
for many operads, including $\S$:
\begin{itemize}
\item $L_{\S}C$ is the general cofree coalgebra over the operad $\S$ ---
here, $C$, is a chain-complex that is not necessarily concentrated
in nonnegative dimensions.
\item $P_{\S}C$ is the \emph{pointed irreducible} cofree coalgebra for
$C$ (see definition~\ref{def:pointedirreducible} and \cite[section~4.2]{Smith:cofree}).
\end{itemize}
In all cases, definition~\ref{def:cofreecoalgebra} implies the existence
of an adjunction
\begin{equation}
\forgetful *:\chaincat\leftrightarrows\ircoalgcat:P_{\S}*\label{eq:cofree-adjunction}
\end{equation}
 where $\forgetful *:\ircoalgcat\to\chaincat$ is the forgetful functor.

\section{\label{sec:Free-operads}Cellular coalgebras}

We recall the following, from \cite[chapter~2, proposition~4.3]{Smith:1994}:
\begin{defn}
The functor 
\[
\cf *:\mathbf{S}\to\coalgcat
\]
from simplicial sets to $\mathfrak{S}$-coalgebras sends a simplicial
set to its chain-complex equipped with an $\mathfrak{S}$-coalgebra
structure defined via acyclic models on the simplices.

A $\S$-coalgebra, $D$, will be staid to be \emph{strongly realizable}
if $D\cong\cf X$ for some pointed, reduced simplicial set $X$.\end{defn}
\begin{rem*}
The $\mathfrak{S}$-coalgebra structure coincides with that used to
define Steenrod operations. More accurately, \cite[chapter~2]{Smith:1994}
shows that the ``higher coproducts'' used to define these operations
are part of an $\mathfrak{S}$-coalgebra structure.

We will also define a complementary functor \end{rem*}
\begin{defn}
\label{def:fc}Define a functor
\[
\fc *:\ircoalgcat\to\spaces
\]
as follows:

If $C\in\ircoalgcat$, define the $n$-simplices of $\fc C$ to be
the $\mathfrak{S}$-coalgebra morphisms
\[
\cf{\Delta^{n}}\to C
\]
where $\Delta^{n}$ is the standard $n$-simplex. Face-operators are
defined by inclusion of faces and degeneracies in a corresponding
fashion.\end{defn}
\begin{rem*}
The normalized chain complex satisfies: 
\begin{equation}
C(\Delta^{n})_{k}=\bigoplus_{\mathbf{n}\twoheadrightarrow\mathbf{k}}\ints=\ints^{{n+1 \choose n-k}}\label{eq:standard-simplex-complex}
\end{equation}
where $\mathbf{n}\twoheadrightarrow\mathbf{k}$ runs over all ordered
surjections
\[
[0,\dots,n]\to[0,\dots,k]
\]

This is similar to the definition of the Dold-Kan functor
\[
\Gamma:\chaincat\to\mathbf{sAb}
\]
to the category of simplicial abelian groups (see \cite[chapter~III, section~2]{Goerss-Jardine}),
the essential difference being that $\fc *$ takes $\S$-coalgebra
structures into account.\end{rem*}
\begin{lem}
\label{lem:fccfadjunction}The functors $\cf *$ and $\fc *$ define
an adjunction
\[
\cf *:\ircoalgcat\leftrightarrows\spaces:\fc *
\]
This gives rise to natural transformations (in the appropriate categories)
\begin{eqnarray*}
u_{X}:X & \to & \fc{\cf X}\\
w_{D}:\cf{\fc D} & \to & D
\end{eqnarray*}
\end{lem}
\begin{proof}
We must show that, for any $X\in\spaces$ and $D\in\ccoalgcat$ we
have a bijection
\[
\hom_{\ccoalgcat}(\cf X,D)=\hom_{\spaces}(X,\fc D)
\]
If we start on the left side, it is not hard to see that every morphism
$\cf X\to D$ can be regarded as a collection of morphisms $\cf{\Delta^{k}}\to D$
for each simplex $\Delta^{k}$ of $X$. This also defines a mapping
$X\to\fc D$. The converse argument is also straightforward.\end{proof}
\begin{prop}
\label{prop:cf(x)-maps-to-fccf}If $X$ is a pointed, reduced simplicial
set the adjunction in lemma~\ref{lem:fccfadjunction} implies the
existence of a coalgebra morphism
\[
\cf{u_{X}}:\cf X\to\cf{\fc{\cf X}}
\]
such that $w_{D}\circ\cf{u_{X}}=1:\cf X\to\cf X$. Given pointed,
reduced simplicial set $Y$, and a morphism of $\S$-coalgebras, 
\[
f:\cf X\to\cf Y
\]
the diagram\begingroup  
\[
\xymatrix{\xyC{40pt}{\cf{\fc{\cf X}}}\ar[r]^{\fc f} & {\cf{\fc{\cf Y}}}\\
{\cf X}\ar[u]^{\cf{u_{X}}}\ar[r]_{f} & {\cf Y}\ar[u]_{\cf{u_{Y}}}
}
\]
\endgroup commutes.\end{prop}
\begin{rem*}
If $D\in\ircoalgcat$, the existence of a morphism $D\to\cf{\fc D}$
splitting the canonical morphism $\cf{\fc D}\to D$ is a \emph{necessary}
condition for $D$ to be strongly realizable.

The (slightly) interesting thing about this diagram is that its vertical
maps depend on the topological realizability of $\cf X$ and $\cf Y$
but the horizontal maps need not be topologically realizable.\end{rem*}
\begin{proof}
If $x\in X$ is a simplex and $f(x)=\sum a_{i}y_{i}$ for simplices
$y_{i}\in Y$, then a simple diagram-chase confirms the conclusion.\end{proof}
\begin{defn}
\label{def:pgam}If $C$ is a chain-complex with augmentation $\epsilon:C\to\integers$,
define $\pgam C$ to the $\Gamma\ker\epsilon$, where $\Gamma*$ is
the Dold-Kan functor (see \cite[chapter~III, section~2]{Goerss-Jardine}).
If $C\in\chaincat$, set $\pgam C=\Gamma C$.\end{defn}
\begin{rem*}
This is just a pointed version of the Dold-Kan functor. 
\end{rem*}
We also need some basic properties of simplicial abelian groups. The
following is a direct consequence of the Dold-Kan construction:
\begin{prop}
\label{prop:free-gives-free-gamma}If $D$ is a $\ints$-free chain-complex,
then $\pgam D$ is a $\ints$-free simplicial abelian group.
\end{prop}
We also need
\begin{defn}
\label{def:moore-complex}If $A$ is a pointed, reduced simplicial
abelian group, $\moore A\in\chaincat$ is its associated Moore complex.\end{defn}
\begin{rem*}
Recall that the very simplices of a simplicial abelian group constitute
a chain-complex --- the Moore complex. The functor $\moore *$ is
a ``forgetful'' functor that forgets the extra structure (i.e.,
face and degeneracy maps) of a simplicial set. If $X$ is a simplicial
set, $\moore{\pz X}=C(X)$, the integral chain complex.
\end{rem*}
The following is probably well-known, but we will use it heavily:
\begin{lem}
\label{lem:gammacorrespondence}If $X\in\spaces$ is a pointed, reduced
simplicial set, there is a natural trivial fibration
\[
\gamma_{X}:\pgam C(X)\to\pz X
\]
and a natural trivial cofibration
\[
\iota_{X}:\pz X\to\pgam C(X)
\]
of simplicial abelian groups, where $C(X)$ is the (integral) chain-complex
and $\pz X=\integers X/\integers*$ --- the pointed version of the
free abelian group functor. In addition, $\gamma_{X}\circ\iota_{X}=1:\pz X\to\pz X$.\end{lem}
\begin{rem*}
It well-known that $\pgam$ defines an equivalence of categories where
$\sab$ is the category of pointed reduced simplicial abelian groups
--- see \cite[chapter III]{Goerss-Jardine}.\end{rem*}
\begin{proof}
Consider the surjective homology equivalence
\[
\integers X=C(X)\to C(X)/DC(X)=NC(X)
\]
where $DC(X)$ is the subcomplex generated by degenerates and $NC(X)$
is the normalized chain-complex. Now, regard $\integers X$ as a chain-complex
and take $\pgam(*)$ of this surjection. We get
\[
\pgam C(X)\to\pgam NC(X)=\pz X
\]
by the Dold-Kan correspondence --- see \cite[chapter III]{Goerss-Jardine},
corollary~2.3, theorem~2.5, and corollary~2.12. The second statement
follows by similar reasoning, using the split inclusion
\[
NC(X)\to C(X)
\]
Both maps clearly preserve the abelian group structure.
\end{proof}
Recall that there is an adjunction
\begin{equation}
C(*):\chaincat\leftrightarrows\spaces:\Gamma*\label{eq:chain-gamma-adjunction}
\end{equation}
where $C(*)$ is the integral chain-complex functor (see \cite[chapter~III]{Goerss-Jardine}).
\begin{defn}
\label{def:pcmap}If $C\in\chaincat$, define $p_{C}:\forgetful{\cf{\pgam C}}\to C$
to be the chain-map that corresponds to the \emph{identity map} $1:\pgam D\to\pgam D$
in \ref{eq:chain-gamma-adjunction}, i.e.
\[
\hom_{\chaincat}(C(\pgam D),D)=\hom_{\spaces}(\pgam D,\pgam D)
\]

The chain-map $p_{C}$ induces a canonical coalgebra-morphism (via
the adjunction in equation~\ref{eq:cofree-adjunction}):
\[
a_{C}:\cf{\pgam C}\to P_{\mathfrak{S}}(C)
\]
where $P_{\S}(C)$ is the pointed, irreducible cofree coalgebra constructed
in \cite{Smith:cofree}.\end{defn}
\begin{rem*}
The \emph{other} map defined by the adjunction in \ref{eq:chain-gamma-adjunction}
is
\[
X\to\pgam\forgetful{\cf X}
\]
which is essentially the Hurewicz map.
\end{rem*}
The maps $p_{C}$ allow us to define \emph{cellular coalgebras}: 
\begin{defn}
\label{def:cellularcoalgebra}If $a_{D}:\cf{\pgam\forgetful D}\to P_{\mathfrak{S}}\forgetful D$
is the map in definition~\ref{def:pcmap}, define
\[
\lcell\forgetful D=\mathrm{im}\,\cf{\pgam\forgetful D}\subset P_{\S}\forgetful D
\]
A $\integers$-free pointed, irreducible $\mathfrak{S}$-coalgebra,
$D$, will be called \emph{cellular} if its classifying map (see definition~\ref{def:cofreecoalgebra}
and equation~\ref{eq:cofree-adjunction}) satisfies 
\[
\beta_{D}:D\to\lcell D\subset P_{\mathfrak{S}}\forgetful D
\]
 The category of cellular coalgebras will be denoted $\ccoalgcat$.
\emph{Morphisms} of cellular coalgebras will simply be $\mathfrak{S}$-coalgebra
morphisms.\end{defn}
\begin{rem*}
Cellular coalgebras are necessarily concentrated in nonnegative dimensions. 
\end{rem*}
Definition~\ref{def:cofreecoalgebra} and the \emph{uniqueness} of
the morphism $\alpha_{f}$ imply that:
\begin{thm}[Rigidity Theorem]
\label{thm:(Rigidity-theorem)} Let $C\in\ccoalgcat$, let $D\in\chaincat$
and let $f:\forgetful C\to D$ be any chain-map. Then there exists
a unique $\mathfrak{S}$-coalgebra morphism $\hat{f}:C\to\lcell D$
that makes the diagram
\[
\xymatrix{{\forgetful C}\ar[r]^{{\forgetful{\hat{f}}}\quad}\ar[rd]_{f} & {\lcell D}\ar[d]^{p_{D}}\\
{} & {D}
}
\]
commute. It follows that we have an adjunction
\[
\forgetful *:\chaincat\leftrightarrows\ccoalgcat:\lcell*
\]
where $\forgetful *:\ccoalgcat\to\chaincat$ is the forgetful functor
(compare this with equation~\ref{eq:cofree-adjunction}).\end{thm}
\begin{proof}
Definition~\ref{def:cofreecoalgebra} implies that there exists a
unique map $\hat{f}:C\to P_{\S}D$ making the diagram 
\[
\xymatrix{{\forgetful C}\ar[r]^{{\forgetful{\hat{f}}}\quad}\ar[rd]_{f} & {\forgetful{P_{\S}D}}\ar[d]^{p_{D}}\\
{} & {D}
}
\]
commute. Here $P_{\S}D$ is the pointed, irreducible cofree coalgebra
cogenerated by $D$ (see \cite{Smith:cofree}). The cellularity of
$C$ implies that the unique map $C\to P_{\S}\forgetful C$ has its
image in $\lcell\forgetful C\subset P_{\S}\forgetful C$ and conclusion
follows from the diagram
\[
\xymatrix{{\lcell\forgetful C}\ar[r]^{a_{\forgetful C}}\ar[d]_{\cf{\pgam f}} & {P_{\S}\forgetful C}\ar[d]^{P_{\S}f}\\
{\lcell D}\ar[r]_{a_{D}} & {P_{\S}D}
}
\]
which shows that there is a coalgebra-morphism $C\to\pgam D$ covering
$f$. The conclusion follows from the fact that its composite with
the inclusion $\lcell D\to P_{\S}D$ is \emph{unique.}
\end{proof}
One of the key ideas in this paper is:
\begin{thm}[Injectivity Theorem]
\label{thm:injectivity-theorem}If $D\in\chaincat$, then the map
of $\S$-coalgebras induced by the $\ints$-linear extension of the
set-map of generators 
\[
\gamma_{D}:\cf{\pgam D}\to P_{\S}\moore{\pgam D}
\]
 is injective.\end{thm}
\begin{proof}
See appendix~\ref{sec:Proof-of-theorem}.
\end{proof}
There are well-known canonical chain-homotopy equivalences $\pi:\moore{\pgam D}\to D$
and $\iota:D\to\moore{\pgam D}$ with $\pi\circ\iota=1:D\to D$ (see
\cite[chapter~III, theorem~2.4]{Goerss-Jardine}). It follows that
the composite 
\[
\cf{\pgam D}\to P_{\S}\moore{\pgam D}\xrightarrow{P_{\S}\pi}P_{\S}D
\]
is a chain-homotopy-monomorphism (see \cite[proposition~4.10]{Smith:model-cats}
for a proof that $P_{\S}\pi$ is a chain-homotopy equivalence).

In some cases, we can say precisely what $\lcell D$ is
\begin{cor}
\label{cor:lcell-restricted-complex}If $D\in\chaincat$ satisfies
$D=\moore{\pgam C}$ for some $C\in\chaincat$, then 
\[
\lcell D=\cf{\pgam C}
\]
\end{cor}
\begin{rem*}
Proposition~2.20 in chapter~III of \cite{Goerss-Jardine} implies
that $\lcell D$ is (unnaturally) chain-homotopy equivalent to
\[
\cf{\prod_{j=1}^{\infty}K(H_{j}(D),j}
\]
where $H_{j}(D)$ is the $j^{\text{th}}$ homology group.\end{rem*}
\begin{proof}
The uniqueness of coalgebra morphisms to a cofree coalgebra implies
that the canonical map $a_{D}:\cf{\pgam D}\to P_{\S}D$ fits into
the commutative diagram 
\[
\xymatrix{{\cf{\pgam D}}\ar[rd]_{a_{D}}\ar[r]^{\cf{\pgam\pi}} & {\cf{\pgam C}}\ar[d]^{\gamma_{D}}\\
{} & {P_{\S}D}
}
\]
where $\pi:D\to C$ is the canonical chain-homotopy equivalence mentioned
above. The conclusion follows from the fact that $\gamma_{D}$ is
injective (by theorem~\ref{thm:injectivity-theorem}).
\end{proof}
This immediately implies
\begin{cor}
\label{cor:classifying-map-space-hurewicz}If $X$ is a pointed, reduced
simplicial set, then
\[
\lcell\forgetful{\cf X}=\cf{\pz X}
\]
It follows that $\cf X$ is cellular and that the classifying map
\[
\beta_{\cf X}:\cf X\to\lcell\forgetful{\cf X}=\cf{\pz X}
\]
is the chain-map induced by the Hurewicz map.

If $\varepsilon:\forgetful{\cf{\pz X}}=\forgetful{\lcell\forgetful{\cf X}}\to\forgetful{\cf X}$
is the cogeneration map, then
\[
\varepsilon\left(\sum_{i}a_{i}\left[\sum_{j}b_{i,j}\cdot\sigma_{i,j}\right]\right)=\sum_{i,j}a_{i}b_{i,j}\sigma_{i,j}
\]
where the $\sigma_{i,j}\in X$ are generators of $\forgetful{\cf X}$,
$[\sum_{j}b_{i,j}\cdot\sigma_{i,j}]\in\pz X$ are generators of $\forgetful{\cf{\pz X}}$,
and the $a_{i},b_{i,j}\in\ints$. It follows that$ $ $\lcell\varepsilon:\lcell\forgetful{\lcell\forgetful{\cf X}}=\forgetful{\cf{\pz^{2}X}}\to\forgetful{\cf{\pz X}}$
is given by
\begin{equation}
\lcell\varepsilon\left(\left[\sum_{i}a_{i}\cdot\left[\sum_{j}b_{i,j}\cdot\sigma_{i,j}\right]\right]\right)=\left[\sum_{i,j}a_{i}b_{i,j}\cdot\sigma_{i,j}\right]\label{eq:lcell-codgeneracy}
\end{equation}
\end{cor}
\begin{rem*}
Among other things, this and definition~\ref{def:cofreecoalgebra}
(or theorem~\ref{thm:(Rigidity-theorem)}) imply that the geometrically-relevant
Hurewicz map
\[
\cf X\to\cf{\pz X}\subset P_{\S}\forgetful{\cf X}
\]
is \emph{uniquely determined} by the coalgebra structure of $\cf X$. 

It is natural to ask what additional information the coalgebra structure
provides. This corollary gives the answer: it determines the chain-level
effect of the \emph{Hurewicz map}. But these two data-points (i.e.,
the chain-complex and the chain-map induced by the Hurewicz map) suffice
to define $\integers^{\bullet}X$ --- the cosimplicial space used
to construct Bousfield and Kan's $\integers$-completion, $\integers_{\infty}X$
(see \cite{Bousfield-Kan}).\end{rem*}
\begin{proof}
Let $NC(X)$ be the normalized chain-complex of $X$ (i.e. degenerates
have been factored out). The Dold-Kan correspondence implies that
\begin{eqnarray*}
\pz X & = & \pgam NC(X)\\
\forgetful{\cf X} & = & \moore{\pz X}
\end{eqnarray*}
so the first statement follows from corollary~\ref{cor:lcell-restricted-complex}. 

The Hurewicz map $h:X\to\pz X$ induces an $\mathfrak{S}$-coalgebra
morphism $\cf X\to\cf{\pz X}$ making the composite
\begin{equation}
\cf X\xrightarrow{\cf h}\cf{\pz X}\xrightarrow{\cf{\iota_{x}}}\cf{\pgam{\forgetful{\cf X}}}\to\lcell\forgetful{\cf X}\label{eq:Hurewicz1}
\end{equation}
an $\mathfrak{S}$-coalgebra morphism. Here $\iota_{X}:\pz X\to\pgam{\forgetful{\cf X}}$
is defined in lemma~\ref{lem:gammacorrespondence}. Since the composite
of this with the inclusion with the inclusion $\cf{\pgam{\forgetful{\cf X}}}\hookrightarrow P_{\S}\forgetful{\cf X}$
is \emph{unique} (see definition~\ref{def:cofreecoalgebra} and equation~\ref{eq:cofree-adjunction}),
it follows that the composite in \ref{eq:Hurewicz1} is the classifying
map of $\cf X$, which must be cellular.
\end{proof}
We can also conclude that:
\begin{cor}
\label{cor:fccfgamma}If $C\in\chaincat$ then $\fc{\lcell C}=\pgam C$.\end{cor}
\begin{proof}
The rigidity theorem (\ref{thm:(Rigidity-theorem)}) implies that
\[
\fc{\lcell C}=\{\hom_{\coalgcat}(\cf{\Delta^{i}},\lcell C)\}=\{\hom_{\chaincat}(C(\Delta^{i}),C)\}
\]
where $i$ runs over simplices of all dimensions. But equation~\ref{eq:standard-simplex-complex}
implies that this is just the definition of $\pgam C$ (see \cite[chapter III]{Goerss-Jardine}).\end{proof}
\begin{defn}
\label{def:cosimplicial}Let:
\begin{enumerate}
\item $\coS$ denote the category of cosimplicial, pointed reduced simplicial
sets, 
\item $\cocell$ denote the category of cosimplicial cellular coalgebras,
\item $\cochaincat$ denote the category of cosimplicial chain-complexes
\item $\cofc *:\cocell\to\coS$ be the result of applying $\fc *$ in each
codimension
\end{enumerate}
\end{defn}
\begin{prop}
The list $(\lcell\forgetful *,\beta_{*},v_{*})$ constitutes a triple
(or monad --- see \cite[section~8.6]{Weibel:homological-algebra})
on $\ccoalgcat$ where, for $C\in\ccoalgcat$:
\begin{enumerate}
\item $\beta_{C}:C\to\lcell\forgetful C$ is the classifying map,
\item $v_{C}=\lcell\epsilon:\lcell(\forgetful{\lcell C})\to\lcell C$, where
$\varepsilon:\lcell C\to C$ is the restriction of the cogeneration
map $\forgetful{P_{\S}C}\to C$ (see definition~\ref{def:cofreecoalgebra}).
\end{enumerate}

Consequently, it defines a functor
\[
Q:\ccoalgcat\to\cocell
\]
(see \cite[section~8.6]{Weibel:homological-algebra}). 

Composing $Q(*)$ with the functor $\fc *$, applied in each codimension,
gives rise to a functor
\[
\qhat:\ccoalgcat\to\coS
\]
to the category of pointed reduced cosimplicial sets.\end{prop}
\begin{rem*}
The cogeneration-map
\[
\varepsilon:\forgetful{P_{\S}C}\to C
\]
is defined by the composite
\[
P_{\S}C\hookrightarrow\prod_{j=0}^{\infty}\homzs j(\rs j,C^{\otimes j})\to\homzs 1(\rs 1,C)=C
\]
see \cite{Smith:cofree}. Here, we follow the convention that $\homzs 0(\rs 1,C^{\otimes0})=\ints$
and $\homzs 1(\rs 1,C)=C$.\end{rem*}
\begin{proof}
The diagram\begingroup \begin{small}
\[
\xyC{25pt}\xymatrix{{\lcell\forgetful{(\lcell\forgetful{(\lcell\forgetful C)}}}\ar[d]_{v_{\lcell C}}\ar@{=}[r] & {\lcell\forgetful{(\lcell\forgetful{\lcell\forgetful C})}}\ar[r]^{\qquad\lcell\forgetful{v_{c}}} & {\lcell\forgetful{\lcell\forgetful C}}\ar[d]^{v_{C}}\\
{\lcell\forgetful{\lcell\forgetful C}}\ar[rr]_{v_{C}} & {} & {\lcell\forgetful C}
}
\]
\end{small}\endgroup commutes by virtue of the Rigidity Theorem (\ref{thm:(Rigidity-theorem)}):
in all composed maps, $C$, remains fixed so the composites are the
unique coalgebra morphisms that lift the cogeneration map $\varepsilon:\lcell C\to C$.
A similar argument shows that the diagram
\[
\xyC{30pt}\xymatrix{{\lcell C}\ar[r]^{\lcell\beta_{C}\quad}\ar@{=}[rd] & {\lcell\forgetful{\lcell C}}\ar[d]_{v_{C}} & {\lcell C}\ar[l]_{\qquad\beta_{\lcell\forgetful C}}\ar@{=}[ld]\\
{} & {\lcell C} & {}
}
\]
commutes, verifying the identities that $(\lcell*,\beta_{*},v_{*})$
must satisfy. The construction of $QC$ is also standard --- see \cite[section~8.6]{Weibel:homological-algebra}
or \cite[chapter VII, section 4]{Goerss-Jardine}. The final statement
is straightforward.
\end{proof}
In the case where our cellular coalgebra is of the form $\cf X$,
we can say more about $Q_{*}$:
\begin{lem}
\label{lem:q-of-space}If $X$ is a pointed, reduced simplicial set,
the cosimplicial cellular coalgebra $Q\cf X$ is has levels $Q^{n}\cf X=\cf{\pz^{n+1}X}$
with 
\begin{enumerate}
\item coface maps
\begin{multline*}
\delta^{i}=\cf{\pz^{i}h\pz^{n-i+1}}:\\
Q^{n}\cf X=\cf{\pz^{n+1}X}\to Q^{n+1}\cf X=\cf{\pz^{n+2}X}
\end{multline*}
for $i=0,\dots n+1$, where $h:X\to\pz X$ is the Hurewicz map, and
\item codegeneracy maps
\begin{multline*}
s^{i}=\cf{\pz^{i}\gamma\pz^{n-i}}:\\
Q^{n+1}\cf X=\cf{\pz^{n+2}X}\to Q^{n}\cf X=\cf{\pz^{n+1}X}
\end{multline*}
for $i=0,\dots,n$, where $\gamma=\lcell\varepsilon$ in equation~\ref{eq:lcell-codgeneracy}.
\end{enumerate}

It follows that 
\[
Q^{\bullet}\cf X=\cf{\ints^{\bullet}X}
\]
where $\ints^{\bullet}X$ is the cosimplicial $\ints$-resolution
of $X$ defined in example~4.1 of \cite[chapter~VII, section~4]{Goerss-Jardine}.

\end{lem}
\begin{rem*}
The cosimplicial space, $\ints^{\bullet}X$, is a variation on the
Bousfield-Kan $\ints$-resolution of $X$.\end{rem*}
\begin{proof}
That $Q^{n}\cf X=\cf{\pz^{n+1}X}$ follows by induction and corollary~\ref{cor:classifying-map-space-hurewicz},
which also implies the statements regarding coface and codegeneracies.
The final statement follows from example~4.1 of \cite[chapter~VII, section~4]{Goerss-Jardine}.\end{proof}
\begin{defn}
\label{def:Hurewiczfunctor}Define the functor, $\realization:\ccoalgcat\to\spaces$,
the category of pointed, reduced simplicial sets by
\[
\realization C=\totalspace{\qhat C}=\hom_{\coS}(\Delta^{\bullet},\qhat C)
\]
 ---where $\totalspace *$ is the total-space functor (see \cite[chapter~VIII]{Goerss-Jardine}
or \cite{Bousfield-Kan}) and $\Delta^{\bullet}$ is the standard
cosimplex. We call $\realization C$ the \emph{Hurewicz realization}
of $C$.\end{defn}
\begin{prop}
\label{prop:hurewicz-realization-fibrant}If $C\in\ccoalgcat$, then
$\realization C$ is a Kan space, and any surjection $f:C\to D$ in
$\ccoalgcat$ induces a fibration
\[
\realization f:\realization C\to\realization D
\]

If is also a homology equivalence, then $\realization f$ is a trivial
fibration. \end{prop}
\begin{proof}
All of the coface maps except for the $0^{\text{th}}$ in $\qhat C$
are morphisms of simplicial abelian groups. It follows that $\qhat C$
is ``group-like'' in the sense of section~4 in chapter~X of \cite{Bousfield-Kan}.
The conclusion follows from proposition~4.9 section~4 in chapter~X
of \cite{Bousfield-Kan}. 

If $f$ is also a homology equivalence, then $\qhat f:\qhat C\to\qhat D$
is a pointwise trivial fibration. The final statement follows from
theorem~4.13 in chapter~VIII of \cite{Bousfield-Kan}, and the fact
that $\Delta^{\bullet}$ is cofibrant in $\coS$.
\end{proof}
When $C$ is topologically realizable, we can say a bit more:
\begin{prop}
\label{prop:canonical-x-hx}If $X$ is a pointed, reduced simplicial
set there exists a natural canonical map
\[
h_{X}:X\to\realization\cf X
\]
natural with respect to maps of simplicial sets. If $Y$ is another
pointed, reduced simplicial set and 
\[
f:\cf X\to\cf Y
\]
 is a morphism of $\S$-coalgebras, the diagram
\begin{equation}
\xymatrix{{\cf X}\ar[r]^{f}\ar[d]_{\cf{h_{X}}} & {\cf Y}\ar[d]^{\cf{h_{Y}}}\\
{\cf{\realization\cf X}}\ar[r]_{\cf{\realization f}} & {\cf{\realization\cf Y}}
}
\label{eq:canonical-x-hx1}
\end{equation}
commutes.\end{prop}
\begin{proof}
The $\ints$-completion functor comes equipped with a natural augmentation
map
\[
\phi_{X}:X\to\ints^{\bullet}X
\]
Since lemma~\ref{lem:q-of-space} implies that $Q^{\bullet}X=\cf{\ints^{\bullet}X}$,
the induces a canonical map
\[
\cf{\phi_{X}}:\cf X\to\cf{\ints^{\bullet}X}=\qhat\cf X
\]
and, applying the $\fc *$ functor gives
\[
\fc{\cf{\phi_{X}}}:\fc{\cf X}\to\fc{\cf{\ints^{\bullet}X}}
\]
and we define $h_{X}$ to be \emph{induced} by the composite (i.e.
the morphism of total spaces induced by this)
\[
X\xrightarrow{u_{X}}\fc{\cf X}\xrightarrow{\fc{\cf{\phi_{X}}}}\fc{\cf{\ints^{\bullet}X}}
\]
Commutativity of diagram~\ref{eq:canonical-x-hx1} follows from proposition~\ref{prop:cf(x)-maps-to-fccf},
which shows that the diagram 
\[
\xymatrix{{\cf X}\ar[d]\ar[r]^{f} & {\cf Y}\ar[d]\\
{\cf{\ints X}}\ar[r]_{\lcell f}\ar[d] & {\cf{\ints Y}}\ar[d]\\
{\vdots} & {\vdots}
}
\]
of cosimplicial cellular coalgebras commutes at the lowest level.
\end{proof}
If $X$ is a pointed, reduced simplicial set, lemma~\ref{lem:q-of-space}
implies $\realization\cf X$ is just $\ints^{\bullet}X$ with its
outer copy of $\ints$ replaced by $\pgam*$. We immediately conclude:
\begin{cor}
\label{cor:coalgebrainducescosimplicial}Let $X$ be a pointed, reduced
simplicial set. Then there exists a natural pointwise trivial fibration
of cosimplicial spaces 
\[
\hat{\gamma}_{X}:\hat{Q}\cf X\to\integers^{\bullet}X
\]
with a pointwise natural trivial cofibration 
\[
\hat{\iota}_{X}:\integers^{\bullet}X\to\hat{Q}\cf X
\]
where $\integers^{\bullet}X$ is the cosimplicial $\ints$-resolution
of $X$ defined in example~4.1 of \cite[chapter~VII, section~4]{Goerss-Jardine}. 

Consequently, 
\begin{enumerate}
\item there exists a natural trivial fibration 
\[
\bar{\gamma}_{X}:\mathcal{H}\cf X\to\zinfty X
\]
that makes the diagram 
\begin{equation}
\xymatrix{{X}\ar[r]^{h_{X}\quad}\ar[rd]_{\phi_{X}} & {\realization\cf C}\ar[d]^{\bar{\gamma}_{X}}\\
{} & {\ints_{\infty}X}
}
\label{eq:coalgebrainducescosimplicial1}
\end{equation}
commute, where $\phi_{X}:X\to\ints_{\infty}X$ is the canonical map
(see 4.2 in \cite[chapter~I]{Bousfield-Kan}).
\item there exists a natural right inverse $\bar{\iota}_{X}:\zinfty X\to\mathcal{H}\cf X$
to $\bar{\gamma}_{X}$ that is a weak equivalence.
\end{enumerate}
It follows that $h_{X}:X\to\realization\cf X$ is an integral homology
equivalence if $X$ is $\ints$-good (in the sense of \cite[chapter~I]{Bousfield-Kan}.\end{cor}
\begin{rem*}
The integral homology equivalence in statement~2 is not necessarily
a weak equivalence of cellular coalgebras. Later papers in this series
will take up the question of the model structure on the category of
cellular coalgebras.\end{rem*}
\begin{proof}
Simply define $\hat{\gamma}_{X}$ and $\hat{\iota}_{X}$ to be $\gamma_{X}$
and $\iota_{X}$, respectively, in each level of $\qhat\cf X$ ---
see lemma~\ref{lem:gammacorrespondence}. The maps 
\[
\hat{\gamma}_{X}:\hat{Q}\cf X\to\integers^{\bullet}X
\]
are all surjective morphisms of group-like cosimplicial spaces and
weak equivalences, so they are \emph{trivial fibrations. }They induce
a trivial fibration of total spaces $\bar{\gamma}_{X}:\mathcal{H}\cf X\to\zinfty X$
and the induced map $\bar{\iota}_{X}:\zinfty X\to\mathcal{H}\cf X$
is an inverse, hence also a weak equivalence. 

The commutativity of diagram~\ref{eq:coalgebrainducescosimplicial1}
immediately follows from the way the map $h_{X}$ was defined in proposition~\ref{prop:canonical-x-hx}.\end{proof}
\begin{cor}
\label{cor:maps-cf-big-diagram}If $X$ and $Y$ are pointed, reduced
simplicial sets, and 
\[
f:\cf X\to\cf Y
\]
 is a morphism of $\S$-coalgebras, $f$ induces a map
\[
\realization f':\ints_{\infty}X\to\ints_{\infty}Y
\]
that makes the diagram\textup{
\[
\xymatrix{{\cf X}\ar[rrr]^{f}\ar[rd]_{\cf{h_{X}}}\ar[dd]_{\cf{\phi_{X}}} &  &  & {\cf Y}\ar[ld]^{\cf{h_{Y}}}\ar[dd]^{\cf{\phi_{Y}}}\\
 & {\cf{\realization\cf X}}\ar[ld]_{\cf{\bar{\gamma}_{X}}}\ar[r]^{\cf{\realization f}} & {\cf{\realization\cf Y}}\ar[rd]^{\cf{\bar{\gamma}_{Y}}}\\
{\cf{\ints_{\infty}X}}\ar[rrr]_{\cf{\realization f'}} &  &  & {\cf{\ints_{\infty}Y}}
}
\]
 }

commute, where 
\[
\phi_{X}:X\to\ints_{\infty}X
\]
 is the canonical map (see 4.2 in \cite[chapter~I]{Bousfield-Kan}). 

If $f$ is surjective and a homology equivalence, then $\realization f$
and $\realization f'$ are weak equivalences.\end{cor}
\begin{rem*}
If we remove one level of $\cf *$, we get a commutative diagram of
\emph{spaces:}
\[
\xymatrix{{X}\ar[rd]_{h_{X}}\ar[dd]_{\phi_{X}} &  &  & {Y}\ar[ld]^{h_{Y}}\ar[dd]^{\phi_{Y}}\\
 & {\realization\cf X}\ar[ld]_{\bar{\gamma}_{X}}\ar[r]^{\realization f} & {\realization\cf Y}\ar[rd]^{\bar{\gamma}_{Y}}\\
{\ints_{\infty}X}\ar[rrr]_{\realization f'} &  &  & {\ints_{\infty}Y}
}
\]
\end{rem*}
\begin{proof}
The map $\realization f'$ is the composite
\[
\ints_{\infty}X\xrightarrow{\bar{\iota}_{X}}\realization\cf X\xrightarrow{\realization f}\realization\cf Y\xrightarrow{\bar{\gamma}_{Y}}\ints_{\infty}Y
\]

The upper quadrilateral commutes by proposition~\ref{prop:canonical-x-hx}.
The left and right triangles commute due to diagram\ref{eq:coalgebrainducescosimplicial1}.
The bottom quadrilateral commutes because $\bar{\gamma}_{X}\circ\bar{\iota}_{X}=1$,
which implies that the outer square also commutes. The final statement
follows from corollary~\ref{cor:coalgebrainducescosimplicial}.
\end{proof}
If spaces are \emph{nilpotent,} we can say a bit more:
\begin{cor}
\label{cor:cfxy-y-nilpotent}Under the hypotheses of corollary~\ref{cor:maps-cf-big-diagram},
if $Y$ is also nilpotent then $\phi_{Y}:Y\to\ints_{\infty}Y$ is
a weak equivalence with a homotopy-inverse, $\phi'$, that fits into
a commutative diagram 
\[
\xymatrix{{X}\ar[rd]_{h_{X}}\ar[dd]_{\phi_{X}} &  &  & {Y}\ar[ld]^{h_{Y}}\ar[dd]_{\phi_{Y}}\\
 & {\realization\cf X}\ar[ld]_{\bar{\gamma}_{X}}\ar[r]^{\realization f} & {\realization\cf Y}\ar[rd]^{\bar{\gamma}_{Y}}\\
{\ints_{\infty}X}\ar[rrr]_{\realization f'} &  &  & {\ints_{\infty}Y}\ar@/_{1pc}/[uu]_{\phi'}
}
\]
where
\begin{enumerate}
\item $h_{Y}:Y\to\realization\cf Y$ is a weak equivalence
\item a morphism of cellular coalgebras, $f:\cf X\to\cf Y$, induces a map
of simplicial sets 
\[
X\xrightarrow{\phi_{X}}\ints_{\infty}X\xrightarrow{\realization f'}\ints Y\xrightarrow{\phi'}Y
\]

\end{enumerate}
\end{cor}
\begin{proof}
The main statement (that $\phi_{Y}$ is a weak equivalence) follows
from proposition~3.5 in chapter~V of \cite{Bousfield-Kan}.\end{proof}
\begin{cor}
\label{cor:cf-classifies-nilpotent}If $X$ and $Y$ are pointed nilpotent
reduced simplicial sets, then $X$ is weakly equivalent to $Y$ if
and only if there exists a morphism of $\S$-coalgebras
\[
f:\cf X\to\cf Y
\]
that is an integral homology equivalence.\end{cor}
\begin{proof}
Clearly, if 
\[
g:X\to Y
\]
 is a weak equivalence, then 
\[
\cf g:\cf X\to\cf Y
\]
 is a morphism of cellular coalgebras that is an integral homology
equivalence.

Conversely, diagram~\ref{eq:coalgebrainducescosimplicial1} in corollary~\ref{cor:coalgebrainducescosimplicial}
implies that 
\begin{eqnarray*}
\cf X & \xrightarrow{\cf{h_{X}}} & \cf{\realization\cf X}\\
\cf Y & \xrightarrow[\cf{h_{Y}}]{} & \cf{\realization\cf Y}
\end{eqnarray*}
are integral homology equivalences. It follows that the homology equivalence
\[
f:\cf X\to\cf Y
\]
 induces a homology equivalence 
\begin{equation}
\cf{\realization f}:\cf{\realization\cf X}\to\cf{\realization\cf Y}\label{eq:cf-classifies-nilpotent1}
\end{equation}
Now corollary~\ref{cor:cfxy-y-nilpotent} implies the existence of
weak equivalences
\begin{eqnarray*}
X & \xrightarrow{h_{X}} & \realization\cf X\\
Y & \xrightarrow[h_{Y}]{} & \realization\cf Y
\end{eqnarray*}
 so that $\realization\cf X$ and $\realization\cf Y$ are both nilpotent.
It follows that integral homology equivalence in \ref{eq:cf-classifies-nilpotent1}
defines a weak equivalence: 
\[
\realization f:\realization\cf X\to\realization\cf Y
\]
The conclusion follows.
\end{proof}
\appendix

\section{\label{sec:Proof-of-theorem}Proof of theorem~\ref{thm:injectivity-theorem}}

We begin with a general result:
\begin{lem}
\label{lem:diagonals-linearly-independent}Let $C$ be a free abelian
group, let 
\[
\hat{C}=\ints\oplus\prod_{i=1}^{\infty}C^{\otimes i}
\]

Let $e:C\to\hat{C}$ be the function that sends $c\in C$ to
\[
(1,c,c\otimes c,c\otimes c\otimes c,\dots)\in\hat{C}
\]
For any integer $t>1$ and any set $\{c_{1},\dots,c_{t}\}\in C$ of
distinct, nonzero elements, the elements 
\[
\{e(c_{1}),\dots,e(c_{t})\}\in\rats\otimes_{\ints}\hat{C}
\]
are linearly independent over $\rats$. It follows that $e$ defines
an injective function
\[
\bar{e}:\ints[C]\to\hat{C}
\]
\end{lem}
\begin{proof}
We will construct a vector-space morphism
\begin{equation}
f:\rats\otimes_{\ints}\hat{C}\to V\label{eq:diagonals-linearly-independent1}
\end{equation}
such that the images, $\{f(e(c_{i}))\}$, are linearly independent.
We begin with the ``truncation morphism''
\[
r_{t}:\hat{C}\to\ints\oplus\bigoplus_{i=1}^{t-1}C^{\otimes i}=\hat{C}_{t-1}
\]
which maps $C^{\otimes1}$ isomorphically. If $\{b_{i}\}$ is a $\ints$-basis
for $C$, we define a vector-space morphism 
\[
g:\hat{C}_{t-1}\otimes_{\ints}\rats\to\rats[X_{1},X_{2},\dots]
\]
by setting
\[
g(c)=\sum_{\alpha}z_{\alpha}X_{\alpha}
\]
where $c=\sum_{\alpha}z_{\alpha}b_{\alpha}\in C\otimes_{\ints}\rats$,
and extend this to $\hat{C}_{t-1}\otimes_{\ints}\rats$ via 
\[
g(c_{1}\otimes\cdots\otimes c_{j})=g(c_{1})\cdots g(c_{j})\in\rats[X_{1},X_{2},\dots]
\]
The map in equation~\ref{eq:diagonals-linearly-independent1} is
just the composite
\[
\hat{C}\otimes_{\ints}\rats\xrightarrow{r_{t-1}\otimes1}\hat{C}_{t-1}\otimes_{\ints}\rats\xrightarrow{g}\rats[X_{1},X_{2},\dots]
\]
It is not hard to see that 
\[
p_{i}=f(e(c_{i}))=1+f(c_{i})+\cdots+f(c_{i})^{t-1}\in\rats[X_{1},X_{2},\dots]
\]
 for $i=1,\dots,t$. Since the $f(c_{i})$ are \emph{linear} in the
indeterminates $X_{i}$, the degree-$j$ component (in the indeterminates)
of $f(e(c_{i}))$ is precisely $f(c_{i})^{j}$. It follows that a
linear dependence-relation
\[
\sum_{i=1}^{t}\alpha_{i}\cdot p_{i}=0
\]
with $\alpha_{i}\in\rats$, holds if and only if
\[
\sum_{i=1}^{t}\alpha_{i}\cdot f(c_{i})^{j}=0
\]
 for all $j=0,\dots,t-1$. This is equivalent to $\det M=0$, where
\[
M=\left[\begin{array}{cccc}
1 & 1 & \cdots & 1\\
f(c_{1}) & f(c_{2}) & \cdots & f(c_{t})\\
\vdots & \vdots & \ddots & \vdots\\
f(c_{1})^{t-1} & f(c_{2})^{t-1} & \cdots & f(c_{t})^{t-1}
\end{array}\right]
\]
 Since $M$ is the transpose of the Vandermonde matrix, we get
\[
\det M=\prod_{1\le i<j\le t}(f(c_{i})-f(c_{j}))
\]
Since $f|C\otimes_{\ints}\rats\subset\hat{C}\otimes_{\ints}\rats$
is \emph{injective,} it follows that this \emph{only} vanishes if
there exist $i$ and $j$ with $i\ne j$ and $c_{i}=c_{j}$. The second
conclusion follows.\end{proof}
\begin{prop}
\label{pro:simplicespropertyS}Let $X$ be a simplicial set with $C=\cf X$
and with coalgebra structure 
\[
\Delta_{n}:RS_{n}\otimes\cf X\to\cf X^{\otimes n}
\]
and suppose $RS_{2}$ is generated in dimension $n$ by $e_{n}=\underbrace{[(1,2)|\cdots|(1,2)]}_{n\text{ terms}}$.
If $x\in C$ is the image of a $k$-simplex, then
\[
\Delta(e_{k}\otimes x)=\epsilon_{k}\cdot x\otimes x
\]
where $\epsilon_{k}=(-1)^{k(k-1)/2}$.\end{prop}
\begin{rem*}
This is just a chain-level statement that the Steenrod operation $\operatorname{Sq}^{0}$
acts trivially on mod-$2$ cohomology. A weaker form of this result
appeared in \cite{Davis:mco}.\end{rem*}
\begin{proof}
For all $k$, let $\Delta^{k}$ denote the standard $k$-simplex,
whose vertices are $\{[0],\dots,[k]\}$ and whose $j$-faces are $\{[i_{0},\dots,i_{j}]\}$,
with $i_{1}<\cdots<i_{j}$, $j\le k$. Let $C$ be the normalized
chain-complex of $\Delta^{k}$ with augmentation
\[
\epsilon:C\to\ints
\]
that maps all $0$-simplices to $1$ and all others to $0$. Recall
that $(\rs 2)_{n}=\ints[\ints_{2}]$ generated by $ $$e_{n}=[\underbrace{(1,2)|\cdots|(1,2)}_{n\text{ factors}}]$.
Let $T$ be the generator of $\ints_{2}$ --- acting on $C\otimes C$
by swapping the copies of $C$.

We assume that $f(e_{i}\otimes C(\Delta^{j}))\subset C(\Delta^{j})\otimes C(\Delta^{j})$
so that 
\begin{equation}
i>j\implies f(e_{i}\otimes C(\Delta^{j}))=0\label{eq:big-diag-condition}
\end{equation}

Define the contracting homotopy on $C(\Delta^{k})$: \textit{\emph{
\begin{equation}
\varphi_{k}([i_{0},\dots,i_{t}]=\left\{ \begin{array}{cc}
(-1)^{t+1}[i_{0},\dots,i_{t},k] & \mathrm{if}\, i_{t}\ne k\\
0 & \mathrm{if}\, i_{t}=k
\end{array}\right.\label{eq:big-diag0}
\end{equation}
It is easy to verify that 
\[
\varphi_{k}\circ d+d\circ\varphi_{k}=1-i\circ\epsilon
\]
where $i:\ints\to C_{0}$ maps $1$ to $[k]$. We extend this to }}$C\otimes C$
via
\[
\Phi_{k}=\varphi_{k}\otimes1+i\circ\epsilon\otimes\varphi_{k}
\]
and use the Koszul convention on signs. Note that $\Phi_{k}^{2}=0$.
As in section~4 of \cite{Smith:1994}, if $e_{0}\in\rs 2$ is the
$0$-dimensional generator, we define
\[
f:\rs 2\otimes C\to C\otimes C
\]
 inductively by
\begin{eqnarray}
f(e_{0}\otimes[i]) & = & [i]\otimes[i]\nonumber \\
f(e_{0}\otimes[0,\dots,k]) & = & \sum_{i=0}^{k}[0,\dots,i]\otimes[i,\dots,k]\label{eq:big-diag1}
\end{eqnarray}
Let $\sigma=\Delta^{k}$ and inductively define
\begin{align*}
f(e_{k}\otimes\sigma) & =\Phi_{k}(f(\partial e_{k}\otimes\sigma)+(-1)^{k}\Phi_{k}f(e_{k}\otimes\partial\sigma)\\
 & =\Phi_{k}(f(\partial e_{k}\otimes\sigma)
\end{align*}
because of equation~\ref{eq:big-diag-condition}. 

Expanding $\Phi_{k}$, we get
\begin{align}
f(e_{k}\otimes\sigma) & =(\varphi_{k}\otimes1)(f(\partial e_{k}\otimes\sigma))+(i\circ\epsilon\otimes\varphi_{k})f(\partial e_{k}\otimes\sigma)\nonumber \\
 & =(\varphi_{k}\otimes1)(f(\partial e_{k}\otimes\sigma))\label{eq:big-diag2}
\end{align}
 because $\varphi^{2}=0$ and $\varphi\circ i\circ\epsilon=0$. 

Noting that $\partial e_{k}=(1+(-1)^{k}T)e_{k-1}\in\rs 2$, we get
\begin{align*}
f(e_{k}\otimes\sigma) & =(\varphi_{k}\otimes1)(f(e_{k-1}\otimes\sigma)+(-1)^{k}(\varphi_{k}\otimes1)\cdot T\cdot f(e_{k-1}\otimes\sigma)\\
 & =(-1)^{k}(\varphi_{k}\otimes1)\cdot T\cdot f(e_{k-1}\otimes\sigma)
\end{align*}
again, because $\varphi_{k}^{2}=0$ and $\varphi_{k}\circ i\circ\epsilon=0$.
We continue, using equation~\ref{eq:big-diag2} to compute $f(e_{k-1}\otimes\sigma)$:
\begin{align*}
f(e_{k}\otimes\sigma)= & (-1)^{k}(\varphi_{k}\otimes1)\cdot T\cdot f(e_{k-1}\otimes\sigma)\\
= & (-1)^{k}(\varphi_{k}\otimes1)\cdot T\cdot(\varphi_{k}\otimes1)\biggl(f(\partial e_{k-1}\otimes\sigma)\\
 & +(-1)^{k-1}f(e_{k-1}\otimes\partial\sigma)\biggr)\\
= & (-1)^{k}\varphi_{k}\otimes\varphi_{k}\cdot T\cdot\biggl(f(\partial e_{k-1}\otimes\sigma)\\
 & +(-1)^{k-1}f(e_{k-1}\otimes\partial\sigma)\biggr)
\end{align*}
If $k-1=0$, then the left term vanishes. If $k-1=1$ so $\partial e_{k-1}$
is $0$-dimensional then equation~\ref{eq:big-diag1} gives $f(\partial e_{1}\otimes\sigma)$
and this vanishes when plugged into $\varphi\otimes\varphi$. If $k-1>1$,
then $f(\partial e_{k-1}\otimes\sigma)$ is in the image of $\varphi_{k}$,
so it vanishes when plugged into $\varphi_{k}\otimes\varphi_{k}$.

In \emph{all} cases, we can write
\begin{align*}
f(e_{k}\otimes\sigma) & =(-1)^{k}\varphi_{k}\otimes\varphi_{k}\cdot T\cdot(-1)^{k-1}f(e_{k-1}\otimes\partial\sigma)\\
 & =-\varphi_{k}\otimes\varphi_{k}\cdot T\cdot f(e_{k-1}\otimes\partial\sigma)
\end{align*}
If $f(e_{k-1}\otimes\Delta^{k-1})=\epsilon_{k-1}\Delta^{k-1}\otimes\Delta^{k-1}$
(the inductive hypothesis), then 
\begin{multline*}
f(e_{k-1}\otimes\partial\sigma)=\\
\sum_{i=0}^{k}\epsilon_{k-1}\cdot(-1)^{i}[0,\dots,i-1,i+1,\dots k]\otimes[0,\dots,i-1,i+1,\dots k]
\end{multline*}
and the only term that does not get annihilated by $\varphi_{k}\otimes\varphi_{k}$
is 
\[
(-1)^{k}[0,\dots,k-1]\otimes[0,\dots,k-1]
\]
 (see equation~\ref{eq:big-diag0}). We get
\begin{align*}
f(e_{k}\otimes\sigma) & =\epsilon_{k-1}\cdot\varphi_{k}\otimes\varphi_{k}\cdot T\cdot(-1)^{k-1}[0,\dots,k-1]\otimes[0,\dots,k-1]\\
 & =\epsilon_{k-1}\cdot\varphi_{k}\otimes\varphi_{k}(-1)^{(k-1)^{2}+k-1}[0,\dots,k-1]\otimes[0,\dots,k-1]\\
 & =\epsilon_{k-1}\cdot(-1)^{(k-1)^{2}+2(k-1)}\varphi[0,\dots,k-1]\otimes\varphi[0,\dots,k-1]\\
 & =\epsilon_{k-1}\cdot(-1)^{k-1}[0,\dots,k]\otimes[0,\dots,k]\\
 & =\epsilon_{k}\cdot[0,\dots,k]\otimes[0,\dots,k]
\end{align*}
where the sign-changes are due to the Koszul Convention. We conclude
that $\epsilon_{k}=(-1)^{k-1}\epsilon_{k-1}$.
\end{proof}
This leads to the proof:
\begin{cor}
\label{cor:simplicial-abelian-injective}If $D\in\chaincat$, then
\[
\gamma_{D}:\cf{\pgam D}\to P_{\S}\moore{\pgam D}
\]
 is injective.\end{cor}
\begin{proof}
Let 
\[
\alpha:\cf{\pgam D}\to P_{\S}\forgetful{\cf{\pgam D}}
\]
 be the coalgebra's classifying map and let $E=\forgetful{\cf{\pgam D}}$.
Then there is a canonical map $c:\forgetful{\cf{\pgam D}}=E\to\moore{\pgam D}$
that induces a bijection
\begin{equation}
\text{simplicial \emph{generators} of }\forgetful{\cf{\pgam D}}\leftrightarrow\text{\emph{elements} of }\moore{\pgam D}\label{eq:moore-bijection}
\end{equation}
 The results of \cite{Smith:cofree} imply that 
\[
P_{\S}E\subset\prod_{n=0}^{\infty}\homzs n(RS_{n},E^{\otimes n})
\]
and the map $a_{\cf X}$ induces a commutative diagram
\begin{equation}
\xymatrix{{\cf{\pgam D}}\ar[r]^{\alpha}\ar[rd]_{a_{\cf{\pz X}}} & {P_{\S}E}\ar[r]\ar[d]^{P_{\S}c} & {\prod_{n=0}^{\infty}\homzs n(RS_{n},E^{\otimes n})}\ar[d]^{\prod_{n=0}^{\infty}\homzs n(1,c^{\otimes n})}\\
{} & {P_{\S}\moore{\pgam D}}\ar[r] & {\prod_{n=0}^{\infty}\homzs n(RS_{n},\moore{\pgam D}^{\otimes n})}
}
\label{eq:proof-commut-dia}
\end{equation}
where we follow the convention that $\homzs 0(\rs 0,E^{0})=\ints$,
$\homzs 1(\rs 1,E)=E$, Let $p_{n}$ be projection to a factor
\[
p_{n}:\prod_{n=0}^{\infty}\homzs n(RS_{n},E^{\otimes n})\to\homzs n(RS_{n},E^{\otimes n})
\]
If $\sigma\in$ is an $m$-simplex defining an element $[\sigma]\in E_{m}$,
proposition~\ref{pro:simplicespropertyS} implies that 
\[
p_{2}\circ\alpha([\sigma])=\epsilon\cdot(e_{m}\mapsto[\sigma]\otimes[\sigma])\in\homzs 2(RS_{2},E\otimes E)
\]
where $\epsilon=\pm1$, depending on the dimension of $\sigma$. 

Let $F_{2}=e_{m}$ and $F_{k}=\underbrace{e_{m}\circ_{1}\cdots\circ_{1}e_{m}}_{k-1\text{ iterations}}\in RS_{k}$
be the operad-composite (see proposition~2.17 of \cite{Smith:1994}).
The fact that operad-composites map to composites of \emph{coproducts}
in a coalgebra implies that 
\[
p_{k}\circ\alpha([\sigma])=\epsilon^{k-1}\cdot(F_{k}\mapsto\underbrace{[\sigma]\otimes\cdots\otimes[\sigma]}_{k\text{ factors}})\in\homzs k(RS_{k},E^{\otimes k})
\]

If $\{\sigma_{1},\dots,\sigma_{t}\}\in\pgam D$ are \emph{distinct}
$m$-simplices representing elements $\{[\sigma_{1}],\dots,[\sigma_{t}]\}\in\moore{\pgam D}$
in the chain-complex, the bijection in \ref{eq:moore-bijection} implies
that $\{c[\sigma_{1}],\dots,c[\sigma_{t}]\}\in\moore{\pgam D}$ are
\emph{also} distinct (although no longer \emph{generators}). 

Their images in $\prod_{n=0}^{\infty}\homzs n(\rs n,\moore{\pgam D}^{\otimes n})$
will have the property that 
\[
p_{k}\circ a_{C}([\sigma_{i}])=\epsilon^{k-1}\cdot(F_{k}\mapsto\underbrace{c[\sigma_{i}]\otimes\cdots\otimes c[\sigma_{i}]}_{k\text{ factors}})\in\homzs k(RS_{k},\moore{\pgam D}^{\otimes k})
\]

Evaluation of elements of $\prod_{n=1}^{\infty}\homzs n(RS_{n},\moore{\pgam D}^{\otimes n})$
on the sequence $(\epsilon\cdot E_{2},\epsilon^{2}\cdot E_{3},\epsilon^{3}\cdot E_{4},\dots)$
gives a homomorphism of $\ints$-modules
\[
h:\prod_{n=0}^{\infty}\homzs n(RS_{n},\moore{\pgam D}^{\otimes n})\to\prod_{n=0}^{\infty}\moore{\pgam D}^{\otimes n}
\]
and $h\circ a_{C}(c_{i})$ is $e(c[\sigma_{i}])$ in lemma~\ref{lem:diagonals-linearly-independent}
(since $\epsilon^{2}=1$). The conclusion follows from lemma~\ref{lem:diagonals-linearly-independent}.
\end{proof}
\bibliographystyle{amsplain}

%%bibliography begins
\providecommand{\bysame}{\leavevmode\hbox to3em{\hrulefill}\thinspace}
\providecommand{\MR}{\relax\ifhmode\unskip\space\fi MR }
\providecommand{\MRhref}[2]{%
  \href{http://www.ams.org/mathscinet-getitem?mr=#1}{#2}
}
\providecommand{\href}[2]{#2}

%% index begins

    \end{document}